\documentclass{amsart}

\usepackage{bm}
\usepackage[usenames]{color}

\newcommand{\Z}{\mathbb{Z}}

\newcommand{\R}{\mathbb{R}}
\newcommand{\Q}{\mathbb{Q}}

\newcommand{\EP}{\mathcal{EP}}
\newcommand{\EZ}{\mathcal{E}_\Z}
\newcommand{\EIP}{\mathcal{EIP}}
\newcommand{\EIPNC}{\mathcal{EIP}nc}

\newtheorem{Thm}{Theorem}
\newtheorem{theorem}[Thm]{Theorem}

\newtheorem{proposition}[Thm]{Proposition}

\newtheorem{lemma}[Thm]{Lemma}

\newtheorem*{remark}{Remark}

\begin{document}

\title{Equations in nilpotent groups}
\author{Moon Duchin, Hao Liang, and Michael Shapiro}
\thanks{MD is partially supported by NSF grants DMS-1207106 and DMS-1255442.
MS wishes to acknowledge support from NIH grant K25 AI079404-05.}
\begin{abstract} We show that there exists an algorithm to decide any single equation in the 
Heisenberg group in finite time.  The method works for all two-step nilpotent groups with rank-one commutator, which includes the higher Heisenberg groups.  
We also prove that the decision problem for systems of equations is unsolvable in all 
non-abelian free nilpotent groups.
\end{abstract}
\maketitle

\section{Introduction}

\subsection{The equation problems}

There are several variants on the {\em equation problem} in groups,
studying the solvability of equations of the form $w=1$, where $w$ is
a word written as a product of constants (fixed group elements) and
variables (with values ranging over group elements).  For instance,
consider the Heisenberg group $H(\mathbb{Z})=\langle a,b : [a,b,a],
[a,b,b] \rangle$.  One could seek values of $x,y,z$ satisfying the
equation $ax^2y^{-1}z^3abyb^{10}yz=1$.  The equation problem is
decidable if there is an algorithm for taking any single equation and
answering {\sc yes} or {\sc no} to the question of whether solutions
exist.  Harder than solving a single equation is to solve a system of
simultaneous equations, and harder than that is to solve a system of
equations and inequations, where ``inequations" are of the form $w\neq
1$.  Let us denote those three problems by $\EP_1$, $\EP$, and $\EIP$.
Note that systems of equations without constants always have the
trivial solution, but if inequations are also allowed then it becomes
meaningful to consider such systems, and we can call this decision
problem $\EIPNC$. 

 In the 1960s and 1970s many papers focused on
effective algorithms to produce solutions to particular equations in
particular groups; see \cite{lyndon} for a survey.  The work of
Makanin changed the terms of study when he showed that $\EIP$ is
decidable in free groups.  This work has been much generalized, and
now $\EIP$ is known to be decidable in virtually free groups,
hyperbolic groups, and certain free products and graph products
(including right-angled Artin groups), by combined work of Rips, Sela,
Dahmani-Guirardel, and
Diekert-Lohrey-Muscholl, in various
combinations.  See \cite{dg} for an overview.

Roman'kov was the first to show that it is not the case that $\EP$ is
decidable in all nilpotent groups, by exhibiting a four-step, rank-six
nilpotent group in which it is undecidable \cite{romankov}.  In the
same paper he also proves that $\EIPNC$ in the Heisenberg
group (or in any free two-step nilpotent group) would imply
Hilbert's Tenth Problem over the rationals, a major open problem in number
theory.\footnote{The statement of the theorem seems to say that $\EIP$ in any $N(2,q)$
is {\em equivalent} to Hilbert's 10th over $\Q$, but we only see 
$\EIPNC\Rightarrow{\rm H10}/\Q$
in the proof.}
 This leaves open (and motivates!) the more granular questions
of which equation problems are solvable in which nilpotent groups.  Up
to now, the sharpest results we have found in the literature are as
follows, where $N(p,q)$ represents the free nilpotent group of step
$p$ and rank $q$, and we adopt the standing convention that $p,q\ge 2$.
All of the undecidability results below are accomplished
by showing that equations in the group can be used to encode
diophantine equations and vice versa, then appealing to the negative
solution to Hilbert's Tenth Problem over the integers.

\medskip

{\bf Single equations}
\begin{itemize}
\item There is an algorithm to decide any single equation in one variable in any $N(2,q)$ (Repin);
\item There is an algorithm to decide any single equation in up to two variables in $H(\Z)=N(2,2)$ (Truss);
\item $\EP_1$ is undecidable in $N(3,q)$ if $q$ is sufficiently large (Repin/Truss).
\end{itemize}

{\bf Systems of equations}
\begin{itemize}
\item $\EP$ is undecidable in $N(2,q)$ if $q$ is sufficiently large (Durnev);
\item $\EP$ is undecidable in $N(3,2)$  (Truss);
\item $\EIP$ is undecidable in all $N(p,q)$ (Ersov).
\end{itemize}

Below, we show that $\EP_1$ is decidable in $N(2,2)$, but that $\EP$
 is undecidable in all $N(p,q)$, i.e., all
non-abelian free nilpotent groups.  The result that any single
equation in $H(\Z)$ is decidable is the first such result in any
(non-virtually-abelian) nilpotent group, as far as we know.  This is accomplished by
reducing this decision problem to solving a single diophantine
quadratic equation in many variables, which is already known to be
decidable.  The method works for a larger class of groups, allowing us to 
decide any single equation in any two-step nilpotent group with
rank-one commutator.  This class includes
all  $\Z$-central extensions of free abelian groups, including the higher Heisenberg groups, 
and also allows the possibility of torsion
as will be explained below.

\subsection{Relationship to number theory and logic}

First, some well-known facts from the theory of equations over $\Z$.   We thank Bjorn Poonen for his explanations.

A system of polynomials $f_1 = \cdots = f_n = 0$ to be solved over $\Z$ (or $\Q$)
is equivalent to the single equation ${f_1}^2 + \cdots + {f_n}^2 = 0$, so a system of polynomials may be solved whenever one
can solve a single equation of twice the maximal degree occurring in the system.
Skolem observed that any polynomial equation can converted to system
of at-most-quadratic equations by  introducing extra variables.
(For example, $y^2 = x^5 + 3$ is equivalent to $u = x^2$, $v = u^2$, $y^2 = xv + 3$.)
Putting these together, we see that an arbitrary system of polynomial equations can be converted into
a single polynomial of degree at most four.  Thus Hilbert's Tenth Problem, asking whether an algorithm exists
to decide if an arbitrary system of diophantine polynomials has a solution, can be reduced to finding solutions to single fourth-degree polynomials.

Let us abbreviate $\EZ(d,n)$ for the problem of finding an integer solution to a single polynomial in $\Z[x]$ of degree $d$
in $n$ variables.  Then $\EZ(1,n)$ and $\EZ(2,n)$ are decidable for all $n$; the $d=1$ case is a linear algebra exercise
and the quadratic case was settled in 1972 by Siegel \cite{siegel}.
$\EZ(1,d)$ is decidable for all $d$ by approximating roots
numerically and checking nearby integers.
One can try to play $d$ and $n$ off of each other to find the boundary of decidability:
it is known that $d=4$ and $n=11$ suffice for undecidability, each paired with an appropriately large value in the other parameter.
  $\EZ(3,2)$ is decidable.
 However, $\EZ(3,n)$ is still an open problem for every $n\ge 3$, and it is an open possibility that it is decidable for all $n$.  Most of this
 is covered in the survey \cite{poonen}, discussing the negative solution to Hilbert's Tenth.

The undecidability results mentioned above for  nilpotent groups all proceed by drawing a  connection
 from the group theory to the number theory:  one shows that any system of quadratic diophantine equations can
 be encoded as a system of equations in $N(p,q)$ such that one system has solutions if and only if the other does.
Therefore undecidability of those
 group-theoretic problems follows from the classical results  in number theory and logic.

\subsection{Nilpotent groups}

Define the nested commutator, generalizing the usual commutator $[a,b]$, by
$$[a,b,c,d]=\left[ [a,b,c],d \right]= \left[ \left[ [a,b],c \right],d \right],$$ and so on.
Then a finitely generated group is called $k$-step {\em nilpotent} if all nested commutators with $k+1$
arguments are trivial, but not all those with $k$ arguments.

With this notation, the standard discrete  {\em Heisenberg group} is the 2-step nilpotent group
given by the presentation
$$H(\mathbb{Z})=\langle a,b : [a,b,a], [a,b,b] \rangle.$$
It sits in the short exact sequence
\begin{equation*}
1\rightarrow \Z \rightarrow
H(\Z) \rightarrow \Z^2 \rightarrow 1,
\end{equation*}
where a generator $c$ of $\Z$ is mapped to $[a,b]$
by the inclusion map $i$.
The second map is the projection  $p: H(\Z) \to H(\Z)/i(\Z)$ where the image is identified with $\Z^2$
by mapping $a$ and $b$ to the standard basis vectors.

Recall that $k$-step nilpotent groups have  lower central series
$$1=G_{k+1}\trianglelefteq G_{k} \trianglelefteq \cdots \trianglelefteq G_2 \trianglelefteq G_1=G,$$
where $G_{i+1}=[G_i,G]$, so that in particular $G_2$ is the usual
commutator subgroup of $G$, and $G_{k}$ is central in $G$. Each
quotient $G_i/G_{i+1}$ is an abelian group which is virtually
$\Z^{d_i}$, and we call $d_i$ the rank of that quotient group.
Recall that the indexing is set up so that $[G_i,G_j]\subseteq
G_{i+j}$.

We do not treat the step-one (abelian) case in this paper because all of these decision
problems that we discuss are solvable in abelian groups.

The {\em free nilpotent group} $N(p,q)$ of step $p\ge 2$ and rank $q\ge 2$ is formed by taking 
$H=F_q$, the free group on $q$ generators, letting $H_{p+1}$ be the  
group in its lower central series, and defining $N(p,q)=H/H_{p+1}$.  In other words, 
it has $q$ generators, only the relations required to make it $p$--step nilpotent.  
These are universal in the sense that any finitely generated nilpotent group is a quotient of an appropriate $N(p,q)$.

\subsection*{Acknowledgments}

We would like to thank Pete Clark, Bjorn Poonen, Martin Davis, and Fran\c{c}ois Dorais for 
helpful explanations of the number theory and logic connections.

\section{Two-step groups}

Consider any two-step nilpotent group $G$ with
rank-one commutator. We will now construct Mal'cev coordinates for the group.  This is entirely standard in 
the torsion-free case, but we take some care to handle torsion.

By the fact that the commutator of a two-step group is abelian,  the 
classification of abelian groups, and the rank assumption, the short exact sequence $$1\to [G,G] \to G \to G/[G,G] \to 1$$ becomes 
$$1\to \Z\oplus \left(\Z_{k_{1}}\oplus\cdots\oplus\Z_{k_{s}}\right)   \to G\to
\Z^n\oplus \left( \Z_{l_{1}}\oplus\cdots\oplus\Z_{l_{r}}\right) \to 1,$$ for
appropriate cyclic groups.  Let $\mathbf a = (a_1,\ldots,a_n),\mathbf
b = (b_1,\ldots,b_r)$ be lifts to $G$ of a basis for $\Z^n$ and
generators for $\Z_{l_i}$, respectively.  Also let $c$ and $\mathbf
d=(d_1,\ldots,d_s)$ be generators of $[G,G]$, so that any word in the
commutator subgroup can be written uniquely as $g=c^\alpha
d_1^{\alpha_1}\cdots d_s^{\alpha_s}$.  Then these $a_i$, $b_i$, $c$,
and $d_i$ form a generating set for $G$ that we will call a {\em
  Mal'cev generating set} (or Mal'cev basis), denoted by $\{\mathbf a,
\mathbf b, c, \mathbf d\}$.  Its relations are completely given by
declaring that $c$ and all $d_i$ are central, that each $d_i^{k_i}=1$,
and by freely choosing the exponents in the expression $c^*d_1^*\cdots
d_s^*$ for each of the $[a_i,a_j] (i<j)$, $[b_i,b_j] (i<j)$, and
$[a_i,b_j]$.  By construction, commuting $a$ letters with $b$ letters
only creates more central $c$ and $d$ letters, so  each
element can be written in the form $g=a_1^*\cdots a_n^* b_1^*\cdots
b_r^* c^* d_1^* \cdots d_s^*$, and this is unique if the 
the $b$ and $d$ exponents are reduced with respect to their modularities.
This gives a normal form, i.e., a bijective correspondence between
group elements and tuples in
$\Z^n\oplus\Z_{l_{1}}\oplus\cdots\oplus\Z_{l_{r}} \oplus \Z\oplus
\Z_{k_{1}}\oplus\cdots\oplus\Z_{k_{s}}$.  (Here we 
identify $\Z_m$ with $\{0,\dots,m-1\}\subset \Z$.)  This tuple is called the
{\em Mal'cev coordinates} with respect to the Mal'cev basis.

Suppose $g,g'\in G$, with Mal'cev coordinates $(\mathbf A,\mathbf B, C, \mathbf D)$ and 
$(\mathbf A',\mathbf B', C', \mathbf D')$, respectively.
Let $gg'$ have coordinates $(\mathbf A'',\mathbf B'', C'', \mathbf D'')$.
To put $gg'$ in normal form, we need only commute the $a_i$ and $b_i$ letters into place and reduce the $b_i$ in the appropriate moduli.
Hence we have $A''_i=A_i+A'_i$ and
$B''_i\equiv B_i+B'_i \pmod{l_i}$.  
Next, $$C''=C+C'-\sum_{i<j} \alpha_{ij} A_i'A_j  -\sum_{i<j} \beta_{ij}B_i'B_j  -\sum_{i,j} \gamma_{ij}A_i'B_j  +\sum_{\{i:B_i+B_i'\ge l_i\}} \epsilon_i,$$
where the $\alpha,\beta,\gamma$ are the exponents of $c$ in the appropriate commutator relations and $\epsilon_i$ is the exponent 
of $c$ in $b_i^{l_i}$.  
 Each $D''_i$ is a similar expression, reduced modulo $k_i$.

\begin{lemma}
There is an algorithm to decide whether there are simultaneous solutions to any system 
of diophantine equations consisting of linear equations and a single quadratic equation.
\end{lemma}

\begin{proof}
Let $\Sigma$ be a  system of linear equations and $Q$ be a quadratic equation in 
the unknowns $y_1,\ldots,y_m$. We can write these out as  
$\Sigma_i: \sum _{j=1}^{m}\alpha_{ij} y_{j} = C_i$ and $Q:f({\mathbf y})=0$, 
where ${\mathbf y}=(y_1,\ldots,y_m)$ and $f$ is a quadratic polynomial.
We will consider the system $\overline\Sigma$ consisting of $\Sigma$ and $Q$.
A solution of $\Sigma$ exists exactly if
 $C_i$ is an integer combination of the $\alpha_{ij}$ for
each $i$.  One easily checks if gcd$(\alpha_{i1},\ldots,\alpha_{im})$ divides  $C_i$. 
If $\Sigma$ has
no solution, then $\overline\Sigma$ has no solution. 
Now suppose $\Sigma$ has a nonempty solution set $S\subset \Z^m$ 
and let ${\mathbf s}\in S$ be a solution.
Consider the corresponding homogeneous system
$\Sigma_0$ whose $i^{\rm th}$ equation is $\sum _{j=1}^{m}\alpha_{ij} y_{j} = 0$, and let $T_0$ and
$S_0$ be its solution spaces in $\R^{m}$ and $\Z^{m}$, respectively, so that $S_0=T_0\cap \Z^m$. 
From the $\alpha_{ij}$, one
can compute an integral basis ${\mathbf v}_1,\ldots, {\mathbf v}_d\in \Z^{m}$ of
$S_0$, so that $S=\{ c_1{\mathbf v}_1 + \cdots + c_d{\mathbf v}_d +{\mathbf s} : c_i\in \Z\}.$
Now $\Sigma$ has a solution if and only if there exists
${\mathbf y}\in S$ such that $f({\mathbf y})=1$, where $f$ is the quadratic polynomial from $Q$.
But $f({\mathbf y})$ is just a quadratic equation in $c_1,\ldots, c_d$, the integer parameters from above,
and a single diophantine quadratic equation is decidable.
\end{proof}

\begin{lemma}
There is an algorithm to decide whether there are simultaneous solutions to any system 
of diophantine equations consisting of arbitrarily many linear equations, arbitrarily many  congruences,
and a single quadratic equation.
\end{lemma}

\begin{proof}
We suppose again that the variables are $y_1,\ldots, y_m$ and in addition to $\Sigma$ and $Q$ from above we add  the system $\Omega$ 
of congruences, with $k^{\rm th}$ equation  $\Omega_k: f_k({\mathbf y}) \equiv 0 \pmod{M_k}$.  
Since $\Z_{M_{k}}$ is finite, we can enumerate all solutions; denote the solution set by $U_k\subset (\Z_{M_k})^m$.
Let us consider the set $U\subset \Z^m$ that simultaneously lifts the $U_k$.  This set is easily constructed 
by the Chinese Remainder Theorem, and it has the form $U=\{r_1{\mathbf u}_1+\cdots+ r_s{\mathbf u}_s +{\mathbf t}\}.$
Therefore the intersection of $U$ and $S$ still has finitely many integer parameters, 
and the full system is decidable as in the previous lemma.
\end{proof}

\begin{theorem} $\EP_1$ is decidable in any two-step nilpotent group with rank-one commutator,
and in particular in $H(\Z)=N(2,2)$.
\end{theorem}

\begin{proof}
Let $G$ be a two-step nilpotent group with rank-one commutator and
let $\{\mathbf a, \mathbf b, c, \mathbf d\}$ be a 
Mal'cev generating set of $G$, as above.
Let $w=1$ be an equation in $G$.
We know that $w=1$ if and only if the corresponding Mal'cev coordinate vector is the zero vector.
For the $\mathbf a$, $\mathbf b$, and $\mathbf d$ coordinates, this reduces to 
finitely many linear equations and finitely many congruences.

The $c$-coordinate equation is nearly a quadratic in the input data, except for the $\epsilon$ terms coming
from cases in the $\mathbf b$ values.
However, there are only finitely many possible $\mathbf b$ values, and hence only finitely many 
solutions $\mathbf B$ of the $\mathbf b$-coordinate equations.  For each of these, the $c$-coordinate
equation becomes quadratic.  We can check each of these finitely many systems using 
the algorithm described in the previous lemma.
\end{proof}

\begin{remark}
What happens if we try to run this argument in a two-step group with a
higher-rank commutator, say of rank $d>1$?  Most of the argument goes
through, but instead of one quadratic equation, we get $d$ quadratic
equations.  Although general systems of quadratic equations are undecidable,
this process might produce systems falling into a special subclass
of quadratic systems, and {\em a priori} this subclass could be decidable.  
\end{remark}


Next we show that systems of equations are undecidable in two-step groups, saving the higher-step case for the next
section.
We note that this encoding scheme bears a strong resemblance to Roman'kov's approach
to $\EIPNC$ over $N(2,q)$ in \cite{romankov}.

\begin{theorem}\label{EP-rank2} $\EP$ is undecidable in $N(2,q)$ for all $q$.  In particular, it
  is undecidable in $H(\Z)$. \end{theorem}

\begin{proof}
Begin with an arbitrary system of diophantine quadratic equations $\hat\Sigma$ in
variables $x_1,\dots,x_n$, whose $i^{\rm th}$ equation is
$$\hat\Sigma_i: \alpha_i+\sum_j \beta_{ij} x_j + \sum_{j,k}
\gamma_{ijk}x_jx_k=0.$$ We will encode this in $G=N(2,q)$ with a system
of equations in twice as many variables and $3q$ additional equations.
 For each variable $x_j$ in
the integer system, we will have variables $y_j,y_j'$ in the group.
We will take the generators of $N(2,q)$ to be $a_1,\dots,a_q$ and let
$c_{i} = [a_j,a_k]$, $j<k$, $i=1,\dots, {q \choose 2} =: t$.  
The $\{c_i\}$ are a basis of the free abelian group $G_2\cong \Z^t$.
For notational convenience, we take $a= a_1$,
$b=a_2$, $c = c_1 = [a,b]$.

Observe that the $\{\mathbf a,\mathbf c\}$ form a Mal'cev basis for $G$, i.e., 
an arbitrary element of $N(2,q)$ can be written uniquely in the
form
$$g = a^A b^B a_3^{m_3} \dots a_q^{m_q} c^C c_2^{n_2}\dots c_t^{n_t}.$$

We build a system of equations $\Sigma$ in the group $G$ having $i^{\rm th}$ equation
$$\Sigma_i: [a,b]^{\alpha_i}\cdot \prod_j [a,y_j']^{\beta_{ij}} \cdot
\prod_{j, k} [y_j,y_k']^{\gamma_{ijk}} =1$$ and the additional $3q$
equations $[a,y_j]=1$, $[b,y_j']=1$,  $[b,y_j]=[y_j',a]$ for
all $1\le j \le q$.

Using the normal form, it is immediate that 
 $[a,y_j]=1\implies y_j = a^{r_j}g$ and $[b,y_j']=1\implies y'_j = b^{r'_j}g'$ for some
 $g, g' \in G_2$.  Since $[b,y_j]=[a,y_j']$, we have $r_j = r'_j$. It now follows that
$[a,y_j']=c^{r_j}$ and $[y_j,y_k']=c^{r_jr_k}$. That lets us simplify 
$\Sigma_i$  to $c^{\alpha_i+\sum_j \beta_{ij} r_j +
  \sum_{j,k} \gamma_{ijk}r_jr_k}=1$, which of course is satisfied
exactly if the exponent is zero.  Thus a solution to $\Sigma$ can
produce a solution to $\hat\Sigma$ by letting $x_j=r_j$, and on the
other hand a solution to $\hat\Sigma$ yields a solution to $\Sigma$ by
taking $y_j=a^{x_j}$, $y_j'=b^{x_j}$.  
\end{proof}

\section{Higher-step groups}

In this section we give a machine to convert back and forth between equation
systems in $N(p,q)$ and equation systems over the integers.
The ability to do this hinges on a simple linearity lemma.  This is directly inspired by Truss's
approach to $N(3,2)$ in \cite{truss}.

\begin{lemma} In a $k$-step nilpotent group $G$, we have
linearity in the last and second-to-last arguments of a $k$-fold commutator:
$$[r_1,r_2,\ldots,r_{k-1},st]=[r_1,r_2,\ldots,r_{k-1},s]\cdot [r_1,r_2,\ldots,r_{k-1},t] \ ; $$
$$[r_1,r_2,\ldots,r_{k-2},st,r_k]= [r_1,r_2,\ldots,r_{k-2},s,r_k] \cdot [r_1,r_2,\ldots,r_{k-2},t,r_k].$$
 \end{lemma}

\begin{proof}
Here are some basic identities about commutators that hold in all groups:
\begin{equation}\label{1} [x,yz] = [x,y] \cdot [y,[x,z]] \cdot [x,z] \end{equation}
\begin{equation}\label{2} [xy,z] = [x,[y,z]] \cdot [y,z] \cdot [x,z] \end{equation}

Now let $R=[r_1,\ldots,r_{k-1}]\in G_{k-1}$ and we must show $[R,st]=[R,s]\cdot [R,t]$.  We have $[R,st]=[R,s]\cdot [s,[R,t]] \cdot [R,t]$ by \eqref{1}, and the middle term
is trivial because $[G_1,[G_{k-1},G_1]]\subset G_{k+1}=1$.  That proves
the first identity asserted in the lemma.

Now, letting $R=[r_1,\ldots,r_{k-2}]\in G_{k-2}$, we must show $[R,st,r]=[R,s,r]\cdot [R,t,r]$.
By \eqref{1}, $[R,st]=[R,s]\cdot [s,[R,t]]\cdot [R,t]$.  Now let $x=[R,s]$ and $y=[s,[R,t]]\cdot[R,t]$.
Then $[R,st,r]=[xy,r]=[x,[y,r]]\cdot [y,r]\cdot [x,r]$ by \eqref{2}.  But $[x,[y,r]] \in [G_{k-1},G_2]=1$,
so we have shown $[R,st,r]=[y,r]\cdot [x,r].$  But our $y$ is itself a product so we can expand $[y,r]$ with \eqref{2}, obtaining three terms, the only surviving one being
$[R,t,r]$.  Since we already know $[x,r]=[R,s,r]$, and since $G_k$ is abelian, we have shown $[R,st,r]=[R,s,r]\cdot [R,t,r]$.
\end{proof}

\newpage
\begin{proposition}\label{EP-rank3}
$\EP$ is undecidable in all free nilpotent groups of step at least three.
\end{proposition}

\begin{proof}
Let $y_j$ be variables in $G=N(p,q)$ with $p\ge 3$, and let $a,b$ be
two of the generators of $G$ in the standard presentation.  Let
$R=[a,b,b,\ldots,b]\in G_{p-2}$, so that $R=a$ if $p=3$.  We will show
that the system $\Sigma$ whose $i^{\rm th}$ equation is
$$\Sigma_i : [R,b,b]^{\alpha_i}\cdot \prod_j [R,b,y_j]^{\beta_{ij}}\cdot \prod_{j,k}  [R,y_j,y_k]^{\gamma_{ijk}} =1$$
has a solution in $G$ if and only if the system $\hat\Sigma$ whose $i^{\rm th}$ equation is
$$\hat\Sigma_i:\alpha_i+\sum_j \beta_{ij} x_j + \sum_{j,k} \gamma_{ijk}x_jx_k$$ has a solution in $\Z$.

If the standard generators of $G$ are $a,b,c_3,\ldots,c_q$, then we
can write $y_j=a^{A_j}b^{B_j}(\prod_{l=3}^q {c_l}^{C_{lj}})h$ for some
$h\in G_2$.  Let us also write $f=[R,b,a]$ and $g=[R,b,b]$, so that
these are among the many basic generators of $G_k$, which is a free
abelian group.  We note that $[R,b,y_j]=f^{A_j}g^{B_j}\prod
[a,b,c_l]^{C_{lj}}$, because $[R,b,q]\in [G_{k-1},G_2]=1$.  Similarly
$$[R,y_j,y_k]=f^{B_jA_k}g^{B_jB_k}\cdot \prod_l [R,b,c_l]^{B_jC_{lk}}\cdot
\prod_{l,m} [R,c_l,c_m]^{C_{lj}C_{mk}}.$$ Since $G_k$ is free abelian,
a word spelled in its basis elements is trivial if and only if all
exponents are zero.  The exponent of $f$ in $\Sigma_i$ is $\sum_j
\beta_{ij}A_j + \sum_{j,k} \gamma_{ijk}B_jA_k$.  The exponent of $g$
in $\Sigma_i$ is $\alpha_i+\sum_j \beta_{ij} B_j + \sum_{j,k}
\gamma_{ijk}B_jB_k$.  The exponents of all other generators of $G_k$
are quadratic polynomials in which every term contains some $C_{ls}$.
We note that the expression coming from the exponent of $g$ is an
the quadratic polynomial appearing in $\hat\Sigma_i$.

A solution to $\Sigma$ in $G$ would produce a solution to $\hat\Sigma$
in $\Z$ by letting $x_j=B_j$; conversely, a solution to $\hat\Sigma$
could be converted to a solution to $\Sigma$ by letting $B_j=x_j$
while all of the $A$ and $C$ values are set to zero.
\end{proof}

Combining Theorem~\ref{EP-rank2} ($p=2$) and Proposition~\ref{EP-rank3} ($p\ge 3$), we have shown that
systems of equations are undecidable in all $N(p,q)$.

\end{document}